\documentclass[11pt]{amsart}

\usepackage[T1]{fontenc}
\usepackage{lmodern}
\usepackage{amsmath,amssymb,amsthm,mathtools}
\usepackage{microtype}
\usepackage{hyperref}
\hypersetup{colorlinks,citecolor=blue,plainpages=false,hypertexnames=false}
\usepackage{comment}
\allowdisplaybreaks

\newtheorem{theorem}{Theorem}[section]
\newtheorem{lemma}[theorem]{Lemma}
\newtheorem{proposition}[theorem]{Proposition}
\newtheorem{corollary}[theorem]{Corollary}
\theoremstyle{remark}
\newtheorem{remark}[theorem]{Remark}

\newcommand{\C}{\mathbb C}
\newcommand{\R}{\mathbb R}
\newcommand{\Q}{\mathbb Q}
\newcommand{\Z}{\mathbb Z}
\newcommand{\PP}{\mathbb{P}}

\newcommand{\Vol}{\operatorname{Vol}}

\newcommand{\sB}{\mathcal{B}}
\newcommand{\sA}{\mathcal{A}}

\hypersetup{pdftitle={Volume comparison and rigidity for holomorphic sectional curvature}}
\newcommand{\Ric}{\mathrm{Ric}}

\newcommand{\RC}{\mathrm{RC}}
\title[K\"ahler volume comparison and rigidity]{Volume comparison and rigidity for positive holomorphic sectional curvature \\ \tiny{- An exposition of a result proved by ChatGPT}}
\author{Ved Datar}
\author{Harish Seshadri}
\date{}

\begin{document}

\begin{abstract}
We prove that the volume of a compact connected K\"ahler manifold with holomorphic sectional curvature at least 2 is bounded above by the volume of the Fubini-Study metric of constant holomorphic sectional curvature 2 on the complex projective space of the same dimension.  Moreover equality holds if and only if the manifold is biholomorphically
isometric to complex projective space. 

This answers a question posed in \cite{XiongYang}. Our approach also yields a different proof of Zhang's sharp volume estimate and Liu's rigidity theorem for compact K\"ahler manifolds with positive Ricci curvature \cite{KZhang}. 

In fact, our main result states that the same sharp volume estimate holds under a new curvature positivity condition (mean RC curvature positivity), which is implied by both positive Ricci curvature and positive holomorphic sectional curvature. The definition of this condition was inspired by the work of Yang \cite{Yang2018}.

 The proofs in this paper are due to ChatGPT 5.6 Sol Pro, and the paper is merely an exposition of its output.  The proofs has been verified by the authors and they take full responsibility for any errors. 
\end{abstract}
\maketitle

\section{Introduction}

All manifolds in this paper are assumed to be connected and without boundary. A basic question in Riemannian geometry is to obtain sharp bounds for the volume of a manifold from curvature bounds. The prototype result is Bishop's theorem: If $(M,g)$ is a compact Riemannian $m$-manifold
satisfying
\[
\Ric_g \geq (m-1)g,
\]
then
\[
\Vol(M,g)\leq \Vol(\mathbb S^m,g_{\mathbb S^m}),
\]
and equality holds precisely when $(M,g)$ is isometric to the round unit sphere.

The corresponding result for  K\"ahler manifolds was proved only relatively recently by Zhang \cite{KZhang}: Let $\omega_{\C\PP^n}$ denote the Fubini-Study metric on $\C\PP^n$ normalized so that 
\begin{equation} \label{norm}
\Ric_{\omega_{\C\PP^n}}=(n+1)\omega_{\C\PP^n}.
\end{equation}
If
$(X,\omega)$ is an $n$-dimensional compact K\"ahler manifold satisfying
\[
\Ric_\omega \geq (n+1)\omega,
\]
then
\[
\Vol(X, \omega)\leq \Vol(\C\PP^n,\omega_{\C\PP^n}) = \frac{(2 \pi)^n}{n!}.
\]

Moreover equality holds precisely when $(X,\omega)$ is biholomorphically
isometric to $(\C\PP^n,\omega_{\C\PP^n})$.

Note that the normalization (\ref{norm}) is equivalent to either of the two conditions 
\[
H_{\omega_{\C\PP^n}} \equiv 2,
\qquad
[\omega_{\C\PP^n}]=2\pi c_1(\mathcal O_{\PP^n}(1)),
\]
where $H_\omega$ denotes holomorphic sectional curvature of $\omega$. 

It is natural to ask whether the same sharp volume bound holds if the Ricci lower
bound is replaced by a lower bound on holomorphic sectional curvature. Positive holomorphic sectional curvature does not
force positivity of Rici curvature and, consequently, the underlying K\"ahler manifold need not be Fano.  the manifold to be Fano. In fact, the latter class is considerably larger than the class of Fano manifolds: recently, Shiyu Zhang proved
that every projective manifold obtained from a projective toric manifold by a
finite sequence of point blow-ups admits a K\"ahler metric with positive
holomorphic sectional curvature \cite{ShiyuZhang}; in particular, every
rational surface does. Hence it is not clear, a priori, that there is {\it any} upper bound on volumes in this setting.

Perhaps surprisingly, it turns out that the same sharp volume estimate holds even for K\"ahler manifolds with positive holomorphic curvature - this is the main result of this paper. In fact, we prove that the sharp estimate holds for {\it mean RC curvature} positivity, a condition implied by both Ricci positivity and positive holomorphic sectional curvature. This condition was inspired by the important work of Yang, who introduced {\it RC-positivity} and proved that it implies projectivity and rational
connectedness \cite{Yang2018,Yang2020}.

\begin{theorem}\label{thm:main}
Let $(X,\omega)$ be a compact $n$-dimensional K\"ahler manifold.
Suppose that
\[
\mu_{\RC}(x)\geq \frac{n+1}{n}
\]
for every $x\in X$. Then
\[
\int_X\omega^n\leq (2\pi)^n.
\]
If equality holds, then the complex scalar curvature $S_\omega$ satisfies
\[
\frac{\displaystyle\int_X S_\omega\,\omega^n}
     {\displaystyle\int_X\omega^n}
=n(n+1).
\]
\end{theorem}

As an immediate consequence we obtain the following sharp comparison theorem.

\begin{corollary}\label{cor:main}
Let $(X,\omega)$ be a compact $n$-dimensional K\"ahler manifold. If
\[
\Ric_\omega\geq(n+1)\omega
\qquad\text{or}\qquad
H_\omega\geq2,
\]
then
\[
\int_X\omega^n\leq(2\pi)^n.
\]
Moreover equality holds if and only if $(X,\omega)$ is biholomorphically
isometric to $(\C\PP^n,\omega_{\C\PP^n})$.
\end{corollary}

In particular, Corollary~\ref{cor:main} solves the volume comparison
problem posed by Xiong and Yang \cite{XiongYang}. In the Ricci case it also recovers Zhang's
optimal volume bound. The proof is different from the arguments used in
\cite{KZhang}; its starting point is instead Yang's RC-positivity method
and a maximum principle for holomorphic cotensors.

We note that in \cite{XiongYang}, Xiong and Yang proved the sharp volume estimate of Theorem \ref{thm:main} for positive holomorphic sectional curvature under the additional assumption that the conjugate radius at some point is at
least $\pi/\sqrt{2}$. 



\subsection{Outline of the proof of Theorem \ref{thm:main}}

We briefly describe the main ideas in the proof. The starting point is Yang's
work on RC-positivity. His method uses a favorable tangent direction together
with the maximum principle to prove vanishing theorems for holomorphic
cotensors. For the volume estimate, we need a quantitative version of this
argument which keeps track of a line-bundle twist.

To do this, we introduce  {\it mean RC curvature}. Instead of choosing a single
tangent direction, we average the curvature against positive semidefinite
endomorphisms of trace one. A finite-dimensional minimax argument gives a
formulation which is well suited to the curvature of tensor powers of the
cotangent bundle. Royden's inequality then shows that
\[
H_\omega\geq 2
\]
implies
\[
\mu_{\RC}\geq \frac{n+1}{n}.
\]
The same lower bound follows directly from
\[
\Ric_\omega\geq (n+1)\omega.
\]

The main analytic consequence is a quantitative vanishing theorem. If
$(L,h)$ is a Hermitian holomorphic line bundle satisfying
\[
\sqrt{-1}\Theta_h(L)\leq \kappa\omega,
\]
then
\[
H^0\left(
X,(\Omega_X^1)^{\otimes m}\otimes L^{\otimes q}
\right)=0
\]
whenever
\[
m>\frac{n\kappa}{n+1}q.
\]
The proof is a maximum-principle argument which uses the mean RC lower bound
to control the different cotensor directions at the same time.

The next step is to convert this vanishing theorem into a bound for the
dimension of a linear system. This is done using a refinement of the
classical Poincar\'e--Siegel jet-counting argument. Roughly speaking, one
takes sufficiently many jets of a basis of sections and uses the natural
filtration of the jet bundle to produce a nonzero twisted cotensor. The
vanishing theorem gives an upper bound for the possible jet weight, while a
counting argument gives a lower bound in terms of the dimension of the space
of sections. Combining the two gives a sharp estimate for
\[
h^0(X,L).
\]
More generally, the same argument gives estimates for
\[
h^0\left(X,(\Omega_X^1)^{\otimes s}\otimes L\right)
\]
with a fixed cotensor factor.

Applying this estimate to large powers of line bundles and using asymptotic
Riemann--Roch gives the volume bound. Since the K\"ahler class need not be
rational, we first approximate it by rational K\"ahler classes and then pass
to the limit.

The estimate with a fixed cotensor factor is also used in the equality case.
When
\[
\int_X\omega^n=(2\pi)^n,
\]
comparison with Hirzebruch--Riemann--Roch determines the next coefficient in
the Hilbert polynomial and gives
\[
c_1(X)\cdot
\left(\frac{[\omega]}{2\pi}\right)^{n-1}=n+1.
\]
Equivalently, the average complex scalar curvature is $n(n+1)$.

The final rigidity arguments are standard. Under $H_\omega\geq2$, Berger's
averaging formula gives the pointwise lower bound
\[
S_\omega\geq n(n+1),
\]
and equality of the average forces $H_\omega\equiv2$. Under the Ricci lower
bound, equality forces
\[
\Ric_\omega=(n+1)\omega.
\]
The corresponding rigidity results then identify the metric with the
Fubini--Study metric on $\PP^n$.

\subsection{Organisation of the paper}

In Section~2 we introduce mean RC curvature, give its minimax
characterization, and prove that both $H_\omega\geq2$ and
$\Ric_\omega\geq(n+1)\omega$ imply
\[
\mu_{\RC}\geq \frac{n+1}{n}.
\]
Section~3 contains the quantitative vanishing theorem for
line-bundle-valued cotensors. In Section~4 we prove the refined
Poincar\'e--Siegel estimate for spaces of sections. Section~5 combines these
estimates with rational approximation and Riemann--Roch to prove the volume
bound and treat the equality case.

\subsection*{Curvature conventions} Write
\[
 \omega=\sqrt{-1}\,g_{i\bar j}\,dz^i\wedge d\bar z^j.
\]
The curvature of $T_X$ is normalized by
\[
 \langle R^{T_X}(u,\bar u)v,w\rangle=R(u,\bar u,v,\bar w).
\]
For a Hermitian line bundle $(L,h)$, write
\[
 \sqrt{-1}\Theta_h(L)
 =\sqrt{-1}\,\Theta_{i\bar j}\,dz^i\wedge d\bar z^j,
 \qquad
 R^L(u,\bar u)=\Theta_{i\bar j}u^i\bar u^j,
\]
so $c_1(L)=[\sqrt{-1}\Theta_h(L)]/(2\pi)$.  With these conventions, a holomorphic section of a
Hermitian holomorphic vector bundle $E$ satisfies
\begin{equation}\label{eq:bochner}
 (\partial\bar\partial|s|^2)(u,\bar u)
 =|\nabla'_u s|^2-\langle R^E(u,\bar u)s,s\rangle.
\end{equation}
\section{Averaged RC positivity}\label{sec:mean-rc}
Let $(X,\omega)$ be a compact $n$-dimensional K\"ahler manifold. Let $E$ be a holomorphic vector bundle with a Hermitian metric $\eta$. Let $R^E \in \Gamma(\Lambda^{1,1}T^*_X\otimes \mathrm{End}(E))$ denote the curvature of the Chern connection. In \cite{Yang2020}, Yang defines $(E,\eta)$ to be RC positive if for all $x\in X$ and all $v\in E_x$, there exists a $u \in T_x^{1,0}X$ such that $$R^E(u,\bar u, v,\bar v)>0.$$ It is called {\em uniformly RC positive} with lower bound $\mu>0$ if for all $x\in X$, there exists a unit vector $u\in T_x^{1,0}X$ such that for all $v\in E_x$, $$R^E(u,\bar u, v,\bar v) \geq \mu|v|^2.$$ Yang then proves some vanishing results that play a key role in his proof of rational connectedness. He also observes that if $H\geq 2$, then one has uniform RC positivity with lower bound $\mu=1$. 

For our volume estimates, we prove analogous vanishing theorems under an averaged version of Yang's positivity, which we call {\em mean uniform RC positivity}. For a Hermitian vector space $V$, we denote by $\mathrm{Herm}^{\geq 0}(V)$ the set of all non-negative definite Hermitian endomorphisms and let $$\sB(V) := \{B\in \mathrm{Herm}^{\geq 0}(V)~|~ \mathrm{tr}(B) = 1\}.$$  Let $\sB_x:= \sB(T_x^{1,0}X).$ For any $x\in X$ and $B\in \sB_x$, we define $$R^E_B:= \sum_{i,j}B^{i\bar j}R^E_{i\bar j} \in \mathrm{End}(E_x),$$ and
\begin{equation}\label{mu}
 \mu_{RC}(E,\eta;\omega)(x):=
 \max_{B\in\sB_x}\min_{0\neq v\in E_x}
 \frac{\langle R^E_Bv,v\rangle_\eta}{|v|_\eta^2}.
\end{equation}
We will use the shorthand $\mu_{RC}(x)$ if
$(E,\eta)=(T^{1,0}X,\omega)$.  

\begin{remark}\label{rem:min-max}
Let $ \sA_x^E:= \sB(E_x)$. For fixed $B\in\sB_x$, the minimum of $\operatorname{tr}(AR_B^E)$ over
$A\in\sA_x^E$ is the least eigenvalue of $R_B^E$. Hence 
$$\mu_{RC}(E,\eta;\omega)(x)=
 \max_{B\in\sB_x} \min_{A \in \sA^E_x} \operatorname{tr}(AR_B^E).$$
Since
$(A,B)\mapsto\operatorname{tr}(AR_B^E)$ is bilinear on the compact convex sets $\sA^E_x$ and $\sB_x$,
the finite-dimensional minimax theorem \cite[Theorem~7, p.~319]{AE} gives
\begin{equation}\label{mm}
 \mu_{RC}(E,\eta;\omega)(x)
 =\min_{A\in\sA_x^E}\max_{B\in\sB_x}\operatorname{tr}(AR_B^E).
\end{equation}
For fixed $A$, the maximum over $\sB_x$ is attained at a rank-one projection.
Consequently,
\[
 \mu_{RC}(E,\eta;\omega)(x)
 =\min_{A\in\sA_x^E}\max_{|u|=1}\operatorname{tr}(AR^E_{u\bar u}).
\]
In particular, since $\sA_x^{T^{1,0}X}=\sB_x$,
\begin{equation}\label{unit}
 \mu_{RC}(x)
 =\min_{A\in\sB_x}\max_{|u|=1}\operatorname{tr}(AR_{u\bar u}).
\end{equation}
\end{remark}
The main observation is the following: 

\begin{lemma}\label{lem:rc-ric-h}
Let $(X,\omega)$ be a $n$-dimensional K\"ahler manifold. If
\[
 \Ric_\omega\geq(n+1)\omega
 \qquad\text{or}\qquad
 H_\omega\geq2,
\]
then $\mu_{RC}(x)\geq(n+1)/n$ for all $x \in X$.
\end{lemma}

\begin{proof}
First suppose that $\Ric_\omega\geq(n+1)\omega$. Fix $x\in X$ and let $B=n^{-1}I\in\sB_x$. Then
\[
 R^{T^{1,0}X}_B=\frac1n\Ric_\omega
 \geq\frac{n+1}{n}\operatorname{Id},
\]
and hence $\mu_{RC}(x)\geq(n+1)/n$ from (\ref{mu}).
Next, suppose $H_\omega \geq 2$. By (\ref{mm}), it is enough to show that for all $A\in \sB_x$, there exists a $B\in \sB_x$ such that 
\begin{equation}\label{eq:H-meanRC}
\mathrm{tr}(AR_B^{T^{1,0}X}) \geq \frac{n+1}{n}.
\end{equation} Let $A\in \sB_x$ be an endomorphism  with eigenvalues $a_1,\cdots,a_n \geq 0$ and $\sum a_i = 1$. Let $\{e_1,\cdots,e_n\}$ be a unitary basis of $T_x^{1,0}X$ diagonalizing $A$. Then by Lemma \ref{lem:royden} (\ref{eq:phase2} below, we have $$\mathrm{tr}(AR_A^{T^{1,0}X})= \sum_{i,j}R_{i\bar i j \bar j} a_ia_j \geq \frac{n+1}{n}.$$ So we have verified the condition \eqref{eq:H-meanRC} by choosing $B = A$.
\end{proof}

\begin{remark}[Relation with Yang's work] In the above language, Yang's uniform positivity is recovered if the outer max in the definition of $\mu_{RC}$ is taken over only those endomorphisms in $\sB_x$ that have rank one. If we call such an invariant $\mu_{RC}^{(1)}$, then what Yang observed is that $H_\omega \geq 2$ implies $\mu_{RC}^{(1)} \geq 1.$ On the other hand since we take a maximum over a larger space of endomorphisms, we obtain a bigger lower bound.  
\end{remark}

\begin{lemma}[Royden's inequality]\label{lem:royden} Let $H_\omega \geq 2$. At a point $x$, let $e_1,\ldots,e_n$ be a unitary basis of $T_x^{1,0}X$ and
let $a_1,\ldots,a_n\geq0$.  Then
\begin{equation}\label{eq:phase1}
 \sum_{i,j}R_{i\bar i j\bar j}a_i a_j
 \geq 
 \left(\sum_i a_i\right)^2+\sum_i a_i^2
 \end{equation}
and hence
\begin{equation}\label{eq:phase2}
 \sum_{i,j}R_{i\bar i j\bar j}a_i a_j
 \geq\frac{n+1}{n}\left(\sum_i a_i\right)^2.
\end{equation}

\end{lemma}
\begin{proof}
If all $a_i$ vanish there is nothing to prove.  Otherwise set
\[
 z_\theta=\sum_i\sqrt{a_i}\,e^{\sqrt{-1}\theta_i}e_i,
 \qquad \theta\in\mathbb T^n,
\]
and let $d\theta$ denote normalized Haar measure on $\mathbb T^n$.
The curvature hypothesis gives
\begin{equation}\label{eq:ztheta}
 R(z_\theta,\overline{z_\theta},z_\theta,\overline{z_\theta})
 \geq2\left(\sum_i a_i\right)^2.
\end{equation}
In the expansion of the left side, the phase factor attached to
$R_{i\bar j k\bar\ell}$ is
$e^{\sqrt{-1}(\theta_i-\theta_j+\theta_k-\theta_\ell)}$.  Its average over $\mathbb T^n$  is
nonzero exactly when the multisets $\{i,k\}$ and $\{j,\ell\}$ agree.  The two
pairings contribute the same term by the K\"ahler symmetries, while the fully
diagonal term is counted twice.  Thus
\begin{equation}\label{eq:phaseexpand}
 \int_{\mathbb T^n}R(z_\theta,\overline{z_\theta},z_\theta,
 \overline{z_\theta})\,d\theta
 =2\sum_{i,j}R_{i\bar i j\bar j}a_i a_j
 -\sum_iR_{i\bar i i\bar i}a_i^2.
\end{equation}
Since $R_{i\bar i i\bar i}=H_\omega(e_i)\geq2$, comparison of
\eqref{eq:ztheta} and \eqref{eq:phaseexpand} proves \eqref{eq:phase1}.
Cauchy--Schwarz gives \eqref{eq:phase2}.
\end{proof}
\begin{remark}
For a complex space form with constant holomorphic sectional curvature two,
\[
 R_{i\bar j k\bar\ell}
 =\delta_{ij}\delta_{k\ell}+\delta_{i\ell}\delta_{kj},
\]
and equality holds in \eqref{eq:phase1}.
\end{remark}

\section{A quantitative vanishing theorem}
The main goal of this section is to prove the following sharp vanishing theorem. 
\begin{theorem}\label{thm:threshold}
Let $(X,\omega)$ be a compact $n$-dimensional K\"ahler manifold satisfying $$\mu_{RC}(x) \geq \frac{n+1}{n},~\forall x\in X,$$ and  let $(L,h)$ be a Hermitian holomorphic line bundle satisfying
\begin{equation}\label{eq:linebound}
 \sqrt{-1}\Theta_h(L)\leq\kappa\omega,
 \qquad \kappa\geq0.
\end{equation}
If integers $m\geq1$, $q\geq0$, satisfy $$m > \frac{n\kappa}{n+1}q,$$ then 
\[
H^0\!\left(X,(\Omega_X^1)^{\otimes m}\otimes L^{\otimes q}\right) = 0.
\]

\end{theorem}

\begin{proof}
We argue by contradiction.  Suppose
\[
 0\neq s\in H^0(X,E),
 \qquad
 E=(\Omega_X^1)^{\otimes m}\otimes L^{\otimes q}.
\]
Let $x$ be a point where $|s|^2$ attains its positive maximum, and rescale
$s$ so that $|s(x)|=1$.  Formula~\eqref{eq:bochner} gives, for every
$u\in T_x^{1,0}X$,
\[
 0\geq(\partial\bar\partial|s|^2)(u,\bar u)
 =|\nabla'_u s|^2-\langle R^E(u,\bar u)s,s\rangle.
\]
Hence
\begin{equation}\label{eq:bochner-max-point}
 \langle R^E(u,\bar u)s,s\rangle\geq0.
\end{equation}

Put $V=T_x^{1,0}X$.  Choose a unit vector
$l\in L_x^{\otimes q}$ and write
\[
 s(x)=T\otimes l,
 \qquad T\in(V^*)^{\otimes m},
 \qquad |T|=1.
\]
For each $\alpha\in\{1,\ldots,m\}$, let
$C_\alpha:V\to(V^*)^{\otimes(m-1)}$ be contraction of $T$ in the
$\alpha$-th slot, and define
\begin{equation}\label{eq:density}
 A=\sum_{\alpha=1}^m C_\alpha^*C_\alpha.
\end{equation}
Then $A$ is Hermitian positive semidefinite and
$\operatorname{tr}A=m$.  Since duality reverses the curvature sign,
\begin{equation}\label{eq:densitycurv}
 \left\langle R^{(V^*)^{\otimes m}}_{u\bar u}T,T\right\rangle
 =-\operatorname{tr}(AR_{u\bar u}).
\end{equation}
Apply Remark~\ref{rem:min-max} (\ref{unit}) to the unit-trace endomorphism $A/m$.
There is a unit vector $u\in V$ such that
\[
 \operatorname{tr}(AR_{u\bar u})\geq m\frac{n+1}{n}.
\]
For this $u$, the line-bundle curvature bound gives
\[
 \langle R^E(u,\bar u)s,s\rangle  = \left\langle R^{(V^*)^{\otimes m}}_{u\bar u}T,T\right\rangle +
\langle R^{L^{\otimes q}}_{u\bar u} l, l \rangle  \leq-m\frac{n+1}{n}+q\kappa.
\]
Together with \eqref{eq:bochner-max-point}, this implies
$m\leq n\kappa q/(n+1)$, contradicting the hypothesis.
\end{proof}
An immediate consequence is the following:
\begin{corollary}\label{cor:cotensor}
Let $(X,\omega)$ be a compact $n$-dimensional K\"ahler manifold satisfying $$\mu_{RC}(x) \geq \frac{n+1}{n},~\forall x\in X.$$ Then for every $m\geq1$,
\[
 H^0\!\left(X,(\Omega_X^1)^{\otimes m}\right)=0.
\]
Consequently we have that $H^{p,0}(X)=0$ for $1\leq p\leq n$. In particular, by Hodge decomposition we also have,
\begin{equation}\label{eq:H2}
 H^2(X,\R)=H^{1,1}(X,\R).
\end{equation}
\end{corollary}

\begin{remark}[Relation with the work of Yang]
 Yang proved that positive holomorphic sectional curvature implies uniform RC-positivity of $T_X$, and used this to obtain qualitative vanishing results for holomorphic cotensors, as well as projectivity and rational connectedness ([8,9]). In the present argument, the qualitative consequence needed later is
$
H^2(X,\mathbb R)=H^{1,1}(X,\mathbb R),
$
as recorded in Corollary 3.2. Theorem 3.1 strengthens the cotensor vanishing to a quantitative statement for tensors twisted by powers of a line bundle: under the curvature bound
$
\sqrt{-1}\Theta_h(L)\le \kappa\omega,
$
a nonzero section of
$
(\Omega_X^1)^{\otimes m}\otimes L^{\otimes q}
$
can exist only if
$
m\le \frac{n\kappa}{n+1},q.
$
This sharp linear bound is the input used in the jet-counting argument of the next section.
\end{remark}

\section{A refined Poincar\'e--Siegel type argument}

The next result contains no curvature hypothesis.  It converts any linear
vanishing threshold for line-bundle-valued cotensors into a simultaneous
bound for complete linear systems with an arbitrary fixed cotensor factor. One can think of it as a refinement of the classical Poincaré–Siegel (or Serre–Siegel, in the compact line-bundle setting) jet-counting argument. The use of exterior products of jets is closely related to the Wronskian construction for linear systems. The additional ingredient here is to retain the total filtration weight of the leading jet determinant and estimate its minimal value by a simplex first-moment argument.

\begin{proposition}[Jet estimate with a fixed cotensor factor]\label{prop:jets}
Let $X$ be a compact complex manifold of dimension $n$, let $L$ be
a holomorphic line bundle, and let $c\geq0$.  Assume that for all integers
$m,q\geq1$,
\begin{equation}\label{eq:abstractthreshold}
 H^0\!\left(X,(\Omega_X^1)^{\otimes m}\otimes L^{\otimes q}\right)\neq0
 \quad\Longrightarrow\quad m\leq cq.
\end{equation}
For an integer $s\geq0$, put
\[
 \mathcal E_s:=(\Omega_X^1)^{\otimes s},
 \qquad
 r_s:=\operatorname{rk}(\mathcal E_s)=n^s,
\]
where $\mathcal E_0=\mathcal O_X$.  Then
\begin{equation}\label{eq:abstracth0}
 h^0(X,\mathcal E_s\otimes L)
 \leq
 \frac{r_s}{n!}
 \left(
  \frac{n+1}{n}(c-s)+\frac{n+1}{2}
 \right)_+^n,
\end{equation}
where $t_+:=\max\{t,0\}$.
\end{proposition}
In particular, the case $s=0$ gives the line-bundle estimate
\[
 h^0(X,L)\leq\frac1{n!}
 \left(\frac{n+1}{n}c+\frac{n+1}{2}\right)^n.
\]
Before we present the proof, we discuss some generalities. Let $E$ be a holomorphic vector bundle on $X$ equipped with a decreasing filtration $\mathcal{F}$ of holomorphic sub-bundles: $$E =  \mathcal F^0E \supset \mathcal F^1E\supset\cdots\supset \mathcal F^rE \supset \mathcal F^{r+1}E = 0.$$ The corresponding graded bundle is defined by $$\mathrm{gr}(\mathcal{F}):= \bigoplus_{p=0}^r \mathcal F^pE\slash \mathcal F^{p+1}E.$$ Given any $\sigma\in H^0(X,E)$ the degree of $\sigma$ (with respect to the filtration) is defined by $$p(\sigma) = \max\{p ~|~ \sigma\in H^0(X,\mathcal F^pE)\}.$$ Then the {\em leading or initial term} of $\sigma$, denoted by $[\sigma] \in H^0(X,\mathrm{gr}(\mathcal{F}))$ is then defined by $i_{p(\sigma)}\circ \pi_{p(\sigma)}\circ \sigma$, where $$\pi_p: \mathcal F^pE\rightarrow \mathcal F^pE\slash \mathcal F^{p+1}E\text{ and }i_p: \mathcal F^pE\slash \mathcal F^{p+1}\hookrightarrow \mathrm{gr}(\mathcal{F}). $$

\vspace{2mm}
\begin{proof}[Proof of the Proposition]
Put \(\mathcal G:=\mathcal E_s\otimes L\), \(V=H^0(X,\mathcal G)\) and \(N=\dim_{\mathbb C}V\). If \(N=0\), the assertion is immediate, so assume \(N>0\). Let \(J^r\mathcal G\) denote the \(r\)-th jet bundle of \(\mathcal G\). There is a natural linear map \[ j^r:V\longrightarrow H^0(X,J^r\mathcal G),\qquad \tau\longmapsto j^r\tau. \] At any \(x\in X\), after choosing local holomorphic coordinates \((z^1,\ldots,z^n)\) centered at \(x\), a local holomorphic frame \(e_1,\ldots,e_{r_s}\) of \(\mathcal E_s\), and a local nonvanishing holomorphic section $l$ of \(L\), and writing \(\tau=\sum_{a=1}^{r_s}f^ae_a\otimes l \), its value at \(x\) is represented by \[ j_x^r\tau= \bigl(f^a(x),\partial_1f^a(x),\ldots,\partial_nf^a(x),\ldots, \partial_If^a(x),\ldots\bigr)_{1\leq a\leq r_s,\ |I|\leq r}. \] We now fix \(x\in X\). The kernels of the maps \[ j_x^r:V\longrightarrow J_x^r\mathcal G \] form a descending sequence of subspaces of \(V\), and their intersection is zero. Indeed, a section belonging to every \(\ker j_x^r\) has all of its derivatives vanishing at \(x\), and hence its Taylor series at \(x\) vanishes identically. It therefore vanishes on a neighborhood of \(x\), and hence, by connectedness, on all of \(X\). Since \(V\) is finite-dimensional, the descending sequence \(\ker j_x^r\) stabilizes, and therefore \(j_x^r\) is injective for all sufficiently large \(r\). Fix such an \(r\). If \(\tau_1,\ldots,\tau_N\) is a basis of \(V\), then \[ \sigma= j^r\tau_1\wedge\cdots\wedge j^r\tau_N \in H^0\!\left(X,\bigwedge^N J^r\mathcal G\right) \] is nonzero, since \[ \sigma(x)= j_x^r\tau_1\wedge\cdots\wedge j_x^r\tau_N\neq0 \] by the injectivity of \(j_x^r\). We next recall the natural filtration on the jet bundle. There are truncation maps \[ \phi_p:J^r\mathcal G\longrightarrow J^p\mathcal G,\qquad 0\leq p\leq r, \] obtained by forgetting the jets of orders \(p+1,\ldots,r\). We set \[ \mathcal F^0J^r\mathcal G:=J^r\mathcal G, \qquad \mathcal F^pJ^r\mathcal G:=\ker\phi_{p-1} \quad (1\leq p\leq r), \qquad \mathcal F^{r+1}J^r\mathcal G:=0. \] This gives a decreasing filtration \[ J^r\mathcal G=\mathcal F^0J^r\mathcal G \supset \mathcal F^1J^r\mathcal G \supset\cdots\supset \mathcal F^rJ^r\mathcal G \supset \mathcal F^{r+1}J^r\mathcal G=0. \] Recall that the jet bundles fit into the standard short exact sequences \cite{Vakil} \[ 0\longrightarrow \operatorname{Sym}^p\Omega_X^1\otimes\mathcal G \longrightarrow J^p\mathcal G \longrightarrow J^{p-1}\mathcal G \longrightarrow0, \qquad p\geq1, \] with \(J^0\mathcal G=\mathcal G\). The map \(J^p\mathcal G\to J^{p-1}\mathcal G\) is the natural truncation map which forgets the derivatives of order \(p\). Its kernel is canonically identified with \[ \operatorname{Sym}^p\Omega_X^1\otimes\mathcal G. \] More explicitly, for \(1\leq p\leq r\), the truncation map \[ \phi_p:J^r\mathcal G\longrightarrow J^p\mathcal G \] restricts to a surjective map \[ \mathcal F^pJ^r\mathcal G \longrightarrow \ker\!\left(J^p\mathcal G\longrightarrow J^{p-1}\mathcal G\right). \] Its kernel is precisely \[ \mathcal F^{p+1}J^r\mathcal G = \ker\!\left(J^r\mathcal G\longrightarrow J^p\mathcal G\right). \] Hence \[ \mathcal F^pJ^r\mathcal G/\mathcal F^{p+1}J^r\mathcal G \simeq \ker\!\left(J^p\mathcal G\longrightarrow J^{p-1}\mathcal G\right) \simeq \operatorname{Sym}^p\Omega_X^1\otimes\mathcal G. \] For \(p=0\), we similarly have \[ \mathcal F^0J^r\mathcal G/\mathcal F^1J^r\mathcal G = J^r\mathcal G/\ker(J^r\mathcal G\to J^0\mathcal G) \simeq J^0\mathcal G=\mathcal G = \operatorname{Sym}^0\Omega_X^1\otimes\mathcal G. \] Consequently, \[ \operatorname{gr}(J^r\mathcal G) \simeq \bigoplus_{p=0}^r \operatorname{Sym}^p\Omega_X^1\otimes\mathcal G. \] The filtration on \(J^r\mathcal G\) induces a decreasing filtration on \(\bigwedge^N J^r\mathcal G\). More precisely, we set \[ \mathcal F^k\!\left(\bigwedge^N J^r\mathcal G\right) := \operatorname{im}\left( \bigoplus_{p_1+\cdots+p_N\geq k} \mathcal F^{p_1}J^r\mathcal G\otimes\cdots\otimes \mathcal F^{p_N}J^r\mathcal G \longrightarrow \bigwedge^N J^r\mathcal G \right), \] where the map is the natural antisymmetrization map. With respect to this induced filtration there is a canonical isomorphism \[ \operatorname{gr}\!\left(\bigwedge^N J^r\mathcal G\right) \simeq \bigwedge^N\operatorname{gr}(J^r\mathcal G). \] Using the description of the associated graded bundle of \(J^r\mathcal G\), this becomes \[ \operatorname{gr}\!\left(\bigwedge^N J^r\mathcal G\right) \simeq \bigwedge^N\left( \bigoplus_{p=0}^r \operatorname{Sym}^p\Omega_X^1\otimes\mathcal G \right). \]
For a direct sum one has the canonical decomposition \[ \bigwedge^N\left( \bigoplus_{p=0}^r E_p \right) \simeq \bigoplus_{\substack{a_0+\cdots+a_r=N\\ 0\leq a_p\leq \operatorname{rk}E_p}} \bigotimes_{p=0}^r\bigwedge^{a_p}E_p. \] Taking \[ E_p:=\operatorname{Sym}^p\Omega_X^1\otimes\mathcal G, \] we obtain \[ \operatorname{gr}\!\left(\bigwedge^N J^r\mathcal G\right) \simeq \bigoplus_{\substack{a_0+\cdots+a_r=N\\ 0\leq a_p\leq r_sb_p}} \bigotimes_{p=0}^r \bigwedge^{a_p} \left(\operatorname{Sym}^p\Omega_X^1\otimes\mathcal G\right), \] where \[ \operatorname{rk}E_p=r_sb_p,\qquad b_p= \operatorname{rk}\left(\operatorname{Sym}^p\Omega_X^1\right) = \binom{n+p-1}{n-1}. \] We now apply the leading-term construction to the nonzero section \[ \sigma\in H^0\!\left(X,\bigwedge^N J^r\mathcal G\right). \] Let $M$ denote the total filtration weight of $\sigma$ with respect to the induced filtration. \[ [\sigma]\in H^0\!\left( X, \operatorname{gr}\!\left(\bigwedge^N J^r\mathcal G\right) \right) \] is nonzero and lies in the degree-\(M\) graded piece. Hence at least one of its components in the above direct-sum decomposition is nonzero. Therefore there exist integers \[ a_0,\ldots,a_r, \qquad \sum_{p=0}^r a_p=N, \qquad 0\leq a_p\leq r_sb_p, \] such that \[ M=\sum_{p=0}^r p\,a_p \] and such that there is a nonzero section of \[ \bigotimes_{p=0}^r \bigwedge^{a_p} \left(\operatorname{Sym}^p\Omega_X^1\otimes\mathcal G\right). \] Since \[ \bigwedge^{a_p} \left(\operatorname{Sym}^p\Omega_X^1\otimes\mathcal G\right) \simeq \bigwedge^{a_p} \left(\operatorname{Sym}^p\Omega_X^1\otimes \mathcal E_s\right) \otimes L^{\otimes a_p}, \] and \[ \sum_{p=0}^r a_p=N, \] we obtain a nonzero section of \[ L^{\otimes N}\otimes \bigotimes_{p=0}^r \bigwedge^{a_p} \left(\operatorname{Sym}^p\Omega_X^1\otimes \mathcal E_s\right). \] Finally, the natural injective maps \[ \bigwedge^{a_p} \left(\operatorname{Sym}^p\Omega_X^1\otimes \mathcal E_s\right) \hookrightarrow \left(\operatorname{Sym}^p\Omega_X^1\otimes \mathcal E_s\right)^{\otimes a_p} \hookrightarrow (\Omega_X^1)^{\otimes (p+s)a_p} \] give a nonzero section \[ 0\neq\tau\in H^0\!\left( X, (\Omega_X^1)^{\otimes(sN+M)}\otimes L^{\otimes N} \right), \qquad M=\sum_{p=0}^r p\,a_p. \]

If $sN+M>0$, hypothesis~\eqref{eq:abstractthreshold}, applied with line exponent
$N$, gives
\begin{equation}\label{eq:Mupper}
 sN+M\leq cN;
\end{equation}
for $sN+M=0$ this is automatic.

It remains to bound $M$ from below.  Locally, the degree-$p$ graded piece 
$\operatorname{Sym}^p\Omega_X^1 \otimes \mathcal E_s \otimes L$ has the $r_sb_p$ basis vectors 
$(dz)^\alpha \otimes e_\ell \otimes l$, \ $1 \le \ell \le r_s$, \ $\vert \alpha \vert =p$.  
We call these the $r_sb_p$ available slots at level $p$. The exterior factor $\wedge^{a_p}$
selects $a_p$ distinct such slots. Enumerate the $r_sb_p$ available slots
at level $p$ by pairs $(\ell,\alpha)$, where
\[
 1\leq\ell\leq r_s,
 \qquad
 \alpha\in\mathbb Z_{\geq0}^n,
 \qquad
 |\alpha|=p.
\]
Assign the $a_p$ selected slots above to distinct
such pairs.  For each $\ell$, let $\mathcal U_\ell$ be the union of the unit
cubes $\alpha+[0,1)^n$ corresponding to the selected pairs with first
coordinate $\ell$, and write
\[
 N_\ell:=|\mathcal U_\ell|,
 \qquad
 \ell_0(y):=y_1+\cdots+y_n.
\]
Then
\[
 \sum_{\ell=1}^{r_s}N_\ell=N,
 \qquad
 \sum_{\ell=1}^{r_s}
 \int_{\mathcal U_\ell}\ell_0(y)\,dy
 =M+\frac n2N.
\]
Among measurable subsets of $\mathbb R_{\geq0}^n$ of volume $v$, the
integral of $\ell_0$ is minimized by the simplex
$\{\ell_0\leq(n!v)^{1/n}\}$.  Hence
\[
 \int_{\mathcal U_\ell}\ell_0(y)\,dy
 \geq
 \frac{n}{n+1}(n!N_\ell)^{1/n}N_\ell.
\]
Since $t\mapsto t^{1+1/n}$ is convex,
\[
 \sum_{\ell=1}^{r_s}N_\ell^{1+1/n}
 \geq
 r_s\left(\frac{N}{r_s}\right)^{1+1/n}.
\]
Therefore
\begin{equation}\label{eq:Mlower}
 \frac{M}{N}
 \geq
 \frac{n}{n+1}
 \left(\frac{n!N}{r_s}\right)^{1/n}
 -\frac n2.
\end{equation}
Combining \eqref{eq:Mupper} and \eqref{eq:Mlower} gives
\[
 \left(\frac{n!N}{r_s}\right)^{1/n}
 \leq
 \frac{n+1}{n}(c-s)+\frac{n+1}{2}.
\]
If $N>0$, the expression on the right must be positive.  Solving for $N$ and
including the case in which that expression is nonpositive proves
\eqref{eq:abstracth0}.
\end{proof}

We obtain the curvature form of the estimate by combining
Proposition~\ref{prop:jets} with Theorem~\ref{thm:threshold}.
\begin{corollary}\label{cor:h0}
Let $(X,\omega)$ be a compact $n$-dimensional K\"ahler manifold satisfying
\[
 \mu_{RC}(x)\geq\frac{n+1}{n}
 \qquad\text{for every }x\in X.
\]
For an integer $s\geq0$, put
\[
 \mathcal E_s=(\Omega_X^1)^{\otimes s},
 \qquad
 r_s=n^s.
\]
If a Hermitian holomorphic line bundle $(L,h)$ satisfies
\[
 \sqrt{-1}\Theta_h(L)\leq\kappa\omega,
\]
then
\begin{equation}\label{eq:h0}
 h^0(X,\mathcal E_s\otimes L)
 \leq
 \frac{r_s}{n!}
 \left(
  \kappa+\frac{n+1}{2}-\frac{n+1}{n}s
 \right)_+^n.
\end{equation}
\end{corollary}

\begin{proof}
Theorem~\ref{thm:threshold} gives \eqref{eq:abstractthreshold} with
$c=n\kappa/(n+1)$.  Substitution in \eqref{eq:abstracth0} proves
\eqref{eq:h0}.
\end{proof}

\begin{remark}[Fubini--Study check]
On $(\mathbb \C\PP^n, \omega_{\C\PP^n})$ one has $H\equiv2$.  For
$L=\mathcal O_{\mathbb P^n}(k)$, take $s=0$ in Corollary~\ref{cor:h0}.
The curvature bound has $\kappa=k$, and \eqref{eq:h0} becomes
\[
 \binom{n+k}{n}\leq\frac1{n!}
 \left(k+\frac{n+1}{2}\right)^n.
\]
This is the arithmetic--geometric mean inequality applied to
$k+1,\ldots,k+n$.  Although the inequality is strict for $n>1$, it is asymptotically sharp to the first two orders (in $k$): the two sides have identical $k^n$ and $k^{n-1}$ 
coefficients. These are precisely the orders used below for the volume inequality and rigidity arguments.
\end{remark}

\section{Putting the pieces together}

\subsection{Proof of the optimal volume bound in Theorem \ref{thm:main}}
Set
\[
 \beta=\frac{[\omega]}{2\pi}.
\]
By Corollary~\ref{cor:cotensor}, $H^2(X,\R)=H^{1,1}(X,\R)$, so we can choose rational $(1,1)$ classes
$\beta_\nu\in H^{1,1}(X,\Q)$ converging to $\beta$.  Let $\omega_\nu$ be the harmonic representative of $2\pi\beta_\nu$ with respect to the fixed K\"ahler metric $\omega$.  These are real forms of type $(1,1)$ and satisfy $\omega_\nu\to\omega$ smoothly.  For all large $\nu$, $\omega_\nu$ is K\"ahler.  Choose
$\varepsilon_\nu\downarrow0$ so that
\begin{equation}\label{eq:omega-upper}
 \omega_\nu\leq(1+\varepsilon_\nu)\omega.
\end{equation}

Choose $m_\nu\in\Z_{>0}$ such that $m_\nu\beta_\nu$ is integral.  By the
Lefschetz $(1,1)$ theorem \cite{Huybrechts}, there is a
holomorphic line bundle $L_\nu$ with
$c_1(L_\nu)=m_\nu\beta_\nu$.  By the $\partial\bar\partial$-lemma \cite{Huybrechts}, there is a Hermitian metric $h_\nu$ on $L_\nu$ such that
\begin{equation}\label{eq:Lnu-curv}
 \sqrt{-1}\Theta(L_\nu)=m_\nu\omega_\nu.
\end{equation} 
For $k\geq1$,
\[
 \sqrt{-1}\Theta(L_\nu^k)
 =km_\nu\omega_\nu
 \leq km_\nu(1+\varepsilon_\nu)\omega.
\]
Corollary~\ref{cor:h0}, with $s=0$ and
$\kappa=km_\nu(1+\varepsilon_\nu)$, gives
\begin{equation}\label{eq:hilbert-upper}
 h^0(X,L_\nu^k)
 \leq\frac1{n!}
 \left(km_\nu(1+\varepsilon_\nu)+\frac{n+1}{2}\right)^n.
\end{equation}
Since $L_\nu$ is ample, asymptotic Riemann--Roch and Serre vanishing \cite{Lazarsfeld} give
\begin{equation}\label{eq:RR}
 h^0(X,L_\nu^k)
 =\frac{c_1(L_\nu)^n}{n!}k^n+O(k^{n-1}).
\end{equation}
Divide \eqref{eq:hilbert-upper} by $k^n$ and let $k\to\infty$.  Comparison
with \eqref{eq:RR} yields
\begin{equation}\label{eq:c1bound}
 c_1(L_\nu)^n\leq\bigl(m_\nu(1+\varepsilon_\nu)\bigr)^n.
\end{equation}
But
\[
 c_1(L_\nu)^n
 =\frac{m_\nu^n}{(2\pi)^n}\int_X\omega_\nu^n.
\]
After canceling $m_\nu^n$ in \eqref{eq:c1bound},
\begin{equation}\label{eq:omeganubound}
 \int_X\omega_\nu^n\leq(2\pi)^n(1+\varepsilon_\nu)^n.
\end{equation}
Letting $\nu\to\infty$,
\begin{equation}\label{eq:H1bound}
 \int_X\omega^n\leq(2\pi)^n.
\end{equation}
Sharpness is immediate because $(\mathbb \C\PP^n, \omega_{\C\PP^n})$ attains equality.

\subsection{Equality and rigidity}
 We first prove the equality assertion in Theorem~\ref{thm:main} and then the rigidity statement in Corollary~\ref{cor:main}.  Throughout
this section,
\[
 c_0:=\frac{n+1}{n},
 \qquad
 S_\omega:=\operatorname{tr}_\omega\Ric_\omega
 =\sum_{i,j=1}^nR_{i\bar i j\bar j}
\]
denotes the complex scalar curvature.  Thus the Riemannian scalar curvature
is $2S_\omega$ with the conventions of this paper.

\begin{lemma}\label{lem:almost-integral}
Let $W$ be a finite-dimensional real vector space and let
$\Lambda\subset W$ be a full lattice.  For every $\alpha\in W$ there are
positive integers $k_j\to\infty$ and $\gamma_j\in\Lambda$ such that
\begin{equation}\label{eq:almost-integral}
 \gamma_j-k_j\alpha\longrightarrow0\qquad\text{in }W.
\end{equation}
\end{lemma}

\begin{proof}
Consider the compact torus $W/\Lambda$ and the cyclic subgroup generated by
the image of $\alpha$.  If this subgroup is finite, then
$k_0\alpha\in\Lambda$ for some $k_0>0$, and one may take
$k_j=jk_0$ and $\gamma_j=k_j\alpha$.  If it is infinite, its compact closure
is an infinite compact subgroup, so the identity is not isolated.  Hence
there are nonzero integers $m_j$, unbounded in absolute value, for which
$m_j\alpha\to0$ in $W/\Lambda$.  Replacing $m_j$ by $|m_j|$ preserves
convergence to zero.  By the definition of convergence in the quotient, one
can choose $\gamma_j\in\Lambda$ such that
$\gamma_j-|m_j|\alpha\to0$ in $W$.  Taking $k_j=|m_j|$ proves
\eqref{eq:almost-integral}.
\end{proof}
We now prove the equality part of Theorem \ref{thm:main}. For the convenience of the reader we repeat the statement. 

\begin{theorem}
\label{thm:equality-average}
Let $(X,\omega)$ be a compact $n$-dimensional K\"ahler manifold satisfying
\[
 \mu_{RC}(x)\geq\frac{n+1}{n}
 \quad\text{for every }x\in X,
 \qquad
 \int_X\omega^n=(2\pi)^n.
\]
Then
\begin{equation}\label{eq:average-scalar-equality}
 \frac{\displaystyle\int_XS_\omega\,\omega^n}
      {\displaystyle\int_X\omega^n}
 =n(n+1).
\end{equation}
Equivalently,
\begin{equation}\label{eq:c1-alpha-equality}
 c_1(X)\cdot \left(\frac{[\omega]}{2\pi}\right)^{n-1}=n+1.
\end{equation}
\end{theorem}

\begin{proof}
Set
\[
 \alpha=\frac{[\omega]}{2\pi}\in H^2(X,\R).
\]
The volume equality is precisely
\begin{equation}\label{eq:alpha-volume-one}
 \alpha^n=1.
\end{equation}
Let
\[
 W=H^2(X,\R),
 \qquad
 \Lambda=\operatorname{im}
 \bigl(H^2(X,\Z)/\mathrm{torsion}\to H^2(X,\R)\bigr).
\]
Note that $\Lambda$ is simply the torsion free part of the Neron-Severi group. By Lemma~\ref{lem:almost-integral}, there are $k_j\to\infty$ and integral
classes $\gamma_j$ such that
\begin{equation}\label{eq:equality-delta}
 \delta_j:=\gamma_j-k_j\alpha\longrightarrow0.
\end{equation}
Corollary~\ref{cor:cotensor} gives
$H^2(X,\R)=H^{1,1}(X,\R)$.  The Lefschetz $(1,1)$ theorem \cite{Huybrechts} therefore
produces a holomorphic line bundle $L_j$ with
$c_1(L_j)=\gamma_j$.

Let $\eta_j$ be the harmonic representative, with respect to $\omega$, of
$2\pi\delta_j$.  Then $\eta_j\to0$ in $C^\infty$.  After changing the
Hermitian metric on $L_j$ by the $\partial\bar\partial$-lemma \cite{Huybrechts}, we may arrange
\begin{equation}\label{eq:equality-line-curvature}
 \sqrt{-1}\Theta(L_j)=k_j\omega+\eta_j.
\end{equation}
Put
\[
 \varepsilon_j=\|\eta_j\|_{\operatorname{op},\omega}.
\]
Then $\varepsilon_j\to0$ and
\begin{equation}\label{eq:equality-line-two-sided}
 (k_j-\varepsilon_j)\omega
 \leq\sqrt{-1}\Theta(L_j)
 \leq(k_j+\varepsilon_j)\omega.
\end{equation}
In particular, $L_j$ is positive for all large $j$.

Fix an integer $s\geq0$ and retain the notation
$\mathcal E_s=(\Omega_X^1)^{\otimes s}$ and $r_s=n^s$.  Choose a Hermitian metric on
$\mathcal E_s\otimes K_X^{-1}$.   It follows
from \eqref{eq:equality-line-two-sided} that
$(\mathcal E_s\otimes K_X^{-1})\otimes L_j$ is Nakano positive for all large $j$.
Kodaira--Nakano vanishing \cite{Huybrechts} therefore gives
\begin{equation}\label{eq:vector-kodaira-vanishing}
 H^q(X,\mathcal E_s\otimes L_j)=0\qquad(q>0).
\end{equation}
Thus $h^0(X,\mathcal E_s\otimes L_j)=\chi(X,\mathcal E_s\otimes L_j)$.

Corollary~\ref{cor:h0}, applied with $L=L_j$, the chosen integer $s$, and
$\kappa=k_j+\varepsilon_j$, gives
\begin{align}
 h^0(X,\mathcal E_s\otimes L_j)
 &\leq
 \frac{r_s}{n!}
 \left(
  k_j+\varepsilon_j
  +\frac{n+1}{2}-\frac{n+1}{n}s
 \right)^n\notag\\
 &=\frac{r_s}{n!}k_j^n
 +\frac{r_s}{(n-1)!}(n+1)
  \left(\frac12-\frac{s}{n}\right)k_j^{n-1}
 +o(k_j^{n-1}).
 \label{eq:vector-section-upper}
\end{align}
On the other hand, Hirzebruch--Riemann--Roch and
\eqref{eq:vector-kodaira-vanishing} give
\begin{equation}\label{eq:vector-HRR}
 h^0(X,\mathcal E_s\otimes L_j)
 =\frac{r_s}{n!}\gamma_j^n
 +\frac{1}{(n-1)!}
 \left(c_1(\mathcal E_s)+\frac{r_s}{2}c_1(X)\right)
 \gamma_j^{n-1}
 +O(k_j^{n-2}),
\end{equation}
with the evident omission of the remainder when $n=1$.  By the splitting principle, for $s\geq1$,
\[
 c_1(\mathcal E_s)=-s n^{s-1}c_1(X),
\]
while $c_1(\mathcal E_0)=0$; hence the following identity holds for every $s\geq0$:
\[
 c_1(\mathcal E_s)+\frac{r_s}{2}c_1(X)
 =r_s\left(\frac12-\frac{s}{n}\right)c_1(X).
\]
Using \eqref{eq:equality-delta} and \eqref{eq:alpha-volume-one} in
\eqref{eq:vector-HRR}, we obtain
\begin{align}
 h^0(X,\mathcal E_s\otimes L_j)
 &=\frac{r_s}{n!}k_j^n\notag\\
 &\quad+
 \frac{r_s}{(n-1)!}
 \left(\frac12-\frac{s}{n}\right)
 c_1(X)\alpha^{n-1}k_j^{n-1}
 +o(k_j^{n-1}).
 \label{eq:vector-HRR-expansion}
\end{align}
Comparison of \eqref{eq:vector-section-upper} and
\eqref{eq:vector-HRR-expansion} yields, for every integer $s\geq0$,
\begin{equation}\label{eq:two-sided-c1-coefficient}
 \left(\frac12-\frac{s}{n}\right)
 \left(c_1(X)\alpha^{n-1}-(n+1)\right)
 \leq0.
\end{equation}
Taking $s=0$ gives
$c_1(X)\alpha^{n-1}\leq n+1$, while taking any integer $s>n/2$ gives the
reverse inequality.  Hence
\[
 c_1(X)\alpha^{n-1}=n+1.
\]

Let $\rho_\omega$ be the Ricci form.  Since
$[\rho_\omega]=2\pi c_1(X)$ and
\[
 \rho_\omega\wedge\omega^{n-1}
 =\frac{S_\omega}{n}\omega^n,
\]
we have
\[
 c_1(X)\alpha^{n-1}
 =\frac{1}{n(2\pi)^n}\int_XS_\omega\,\omega^n.
\]
Together with $\int_X\omega^n=(2\pi)^n$, this proves
\eqref{eq:average-scalar-equality}.
\end{proof}

We now extract the two geometric rigidity statements for which the average
scalar curvature has a pointwise lower bound.  In the Ricci case, after
obtaining the K\"ahler--Einstein equation, we use Fujita's optimal
anticanonical-volume theorem and Bando--Mabuchi uniqueness.

\begin{corollary}[Rigidity under holomorphic sectional or Ricci curvature]
\label{cor:equality-rigidity}
Let $(X,\omega)$ be a compact $n$-dimensional K\"ahler manifold satisfying
\[
 \int_X\omega^n=(2\pi)^n.
\]

\begin{enumerate}
\item If $H_\omega\geq2$, then $H_\omega\equiv2$, and
$(X,\omega)$ is biholomorphically isometric to
$(\mathbb \C\PP^n, \omega_{\C\PP^n})$.

\item If $\Ric_\omega\geq(n+1)\omega$, then
$\Ric_\omega=(n+1)\omega$, and again
$(X,\omega)$ is biholomorphically isometric to
$(\mathbb \C\PP^n, \omega_{\C\PP^n})$; in particular,
$H_\omega\equiv2$.
\end{enumerate}
\end{corollary}

\begin{proof}
Both curvature hypotheses imply $\mu_{RC}\geq(n+1)/n$ by
Lemma~\ref{lem:rc-ric-h}, so Theorem~\ref{thm:equality-average} applies.

Assume first that $H_\omega\geq2$.  Then by Berger's averaging trick \cite{Ber},  at each $x\in X$, one has
\begin{equation}\label{eq:sphere-average-H}
 \int_{|\xi|=1}H_\omega(\xi)\,d\sigma_x(\xi)
 =\frac{2S_\omega(x)}{n(n+1)},
\end{equation}
where $d\sigma_x$ denotes the
normalized Lebesgue measure on the unit sphere in $T_x^{1,0}X$. Thus $H_\omega\geq2$ implies
$S_\omega\geq n(n+1)$ pointwise. On the other hand, the integral average of $S_\omega$ is exactly $n(n+1)$ by Theorem~\ref{thm:equality-average}; hence
$S_\omega\equiv n(n+1)$ and $H_\omega\equiv2$. The standard Cartan-Ambrose-Hicks argument then gives the first assertion (cf. \cite{KN2}).

Now assume $\Ric_\omega\geq(n+1)\omega$.  Then
\[
 S_\omega=\operatorname{tr}_\omega\Ric_\omega
 \geq n(n+1).
\]
The average scalar-curvature equality gives
$S_\omega\equiv n(n+1)$.  This in turn forces $\omega$ to be Einstein, i.e., \begin{equation}\label{eq:ricci-equality-independent-start}
 \Ric_\omega=(n+1)\omega.
\end{equation}
Fujita's theorem on optimal volume of K\"ahler-Einstein Fano manifolds \cite[Theorem 1.1]{Fujita} and the uniqueness of K\"ahler-Einstein metrics on $\C\PP^n$ then imply rigidity.  
\end{proof}

\section{Concluding remarks}

(1) Our initial prompt to ChatGPT 5.6 Sol Pro was to find an upper bound for the volumes of compact K\"ahler surfaces with $H \ge 2$. Its proof was not the one in this paper - it proved a non-optimal upper bound for the first eigenvalue of the Laplacian and used that to estimate volume. Subsequent queries yielded a non-optimal estimate in all dimensions based on a proof broadly similar to the current one. Further interactions led to it realizing that the non-optimality arose from the vanishing theorem. Its proof of the revised vanishing theorem led us to seek a general positivty condition implied by both Ricci positiivty and positive holomorphic sectional curvature. 
\vspace{2mm}

(ii) ChatGPT 5.6 Sol Pro was unable to either prove or construct a counterexample to rigidity in the general setting of Theorem \ref{thm:main}. We pose this as an open question: 

{\it Let $(X,\omega)$ be a compact $n$-dimensional K\"ahler manifold.
Suppose that
\[
\mu_{\RC}(x)\geq \frac{n+1}{n}
\]
for every $x\in X$. If
\[
\int_X\omega^n\ = (2\pi)^n,
\]
is $(X, \omega)$ holomorphically isometric to $(\C\PP^n, \omega_{\C\PP^n})$?

\end{document}